\documentclass[12pt,a4paper]{amsart} % Document Class
\usepackage{mathtools}
\DeclarePairedDelimiter\floor{\lfloor}{\rfloor}
\usepackage{orcidlink}
\usepackage{enumitem}
\usepackage[none]{hyphenat}
\usepackage{hyperref}
\usepackage{doi}
\newtheorem{theorem}{Theorem}[section]
\newtheorem{lemma}[theorem]{Lemma}

\newtheorem{corollary}[theorem]{Corollary}
\theoremstyle{definition}
\newtheorem{definition}[theorem]{Definition}
\newtheorem{remark}{Remark}
\newtheorem{example}[theorem]{Example}
\numberwithin{equation}{section}
\title[On some bounds for cardinal invariants of $Q_p(X,Y)$]{On some bounds for cardinal invariants of the space of quasicontinuous functions with the topology of pointwise convergence}

\author[Mandeep]{Mandeep Kumar$^{1}$\thinspace\orcidlink{0000-0001-8773-0927}}
\address{$^1$Department of Mathematics, University of Delhi; Delhi, India, 110007.}

\email{mjakhar5@gmail.com}

\author[Tyagi]{Brij Kishore Tyagi$^{2,\ast}$\thinspace\orcidlink{0000-0003-2660-2432}}
\address{$^2$Department of Applied Mathematics, Delhi Technological University, Delhi, India, 110042.}
\email{brijkishore.tyagi@gmail.com}
\thanks{*Corresponding author.}
\begin{document}
\emergencystretch=1em
%%%%%%%%%%%%%%%%%%%%%%%%%%%%%%%%%%%%%%%%%%%%%%%%%%%%%%%%%%%%%%%							Abstract
%%%%%%%%%%%%%%%%%%%%%%%%%%%%%%%%%%%%%%%%%%%%%%%%%%%%%%%%%%%%%%
\begin{abstract}
In this paper, we obtain some bounds for cardinal invariants such as character, pseudocharacter, tightness, {L}indel\"of number, hereditary {L}indel\"of number, hereditary density and spread of $Q_p(X,Y)$, the space of all quasicontinuous functions from $X$ to $Y$ equipped with the topology of pointwise convergence. We prove that the tightness, hereditary {L}indel\"of number, hereditary density and spread of the space $Q_p(\mathbb{R},\mathbb{R})$ are equal to the cardinality of $\mathbb{R}$.
\end{abstract}
%------------------------------------------------------------%
\maketitle

%%%%%%%%%%%%%%%%%%%%%%%%%%%%%%%%%%%%%%%%%%%%%%%%%%%%%%%%%%%%%%
%               Section: Introduction
%%%%%%%%%%%%%%%%%%%%%%%%%%%%%%%%%%%%%%%%%%%%%%%%%%%%%%%%%%%%%%
\section{Introduction}\label{sec1}

In \cite{Baire1899}, {B}aire  used the condition of quasicontinuity in the study of continuity points of separately continuous functions from $\mathbb{R}^2$ to $\mathbb{R}$. Later, {K}empisty \cite{Kempisty1932} introduced quasicontinuity as a weaker form of continuity for real-valued functions. Further, the properties of quasicontinuous functions were studied in many papers (see, for example, \cite{MR995972}). Quasicontinuous functions were found to have applications in various areas of mathematics, such as topological groups, dynamical systems, and the theory of minimal usco and minimal cusco maps.

%------------------------------------------------------------%
Some examples of quasicontinuous functions that are not continuous include the doubling function $D:[0,1)\to [0,1)$ defined by $D(x)=2x\pmod{1}$, the extended $\sin(1/x)$ function $f:\mathbb{R}\to\mathbb{R}$ defined by
	\[ f(x)= \begin{cases}
		\sin\left(\frac{1}{x}\right) & \text{if $x\neq 0$} \\
		0 & \text{if $x=0$},\\
	\end{cases} \]
the floor function from $\mathbb{R}$ to $\mathbb{R}$ defined by $\floor{x}=\max\{n\in\mathbb{Z}:n\leq x\}$ and the monotone real-valued functions that are left-continuous or right-continuous defined on $\mathbb{R}$.

%------------------------------------------------------------%
The set of all real-valued quasicontinuous functions defined on a topological space $X$ equipped with the topology of pointwise convergence, denoted by $Q_p(X,\mathbb{R})$, was studied in \cite{MR2934374,MR2854741,MR4116833,MR2660032,MR4409699}. The pointwise convergence of sequences of real-valued quasicontinuous functions defined on a {B}aire space was examined by {B}ors\'ik et al.\@ \cite{MR2854741}  and {H}ol\'{a} and {H}ol\'{y} \cite{MR2934374}. {H}ol\'{y} and {M}atej\'{\i}\v{c}ka \cite{MR2660032} studied the  metrizability, first countability, and closed and compact subsets of the space $Q_p(X,\mathbb{R})$. Some cardinal invariants of the space $Q_p(X,\mathbb{R})$ were investigated  by {H}ol\'{a} and {H}ol\'{y} \cite{MR4116833}. Furthermore, {K}umar and {T}yagi \cite{MR4491569} investigated certain cardinal invariants and special mappings such as restriction and induced maps for $Q_p(X,Y)$, the space of all quasicontinuous functions from a space $X$ to a space $Y$ equipped with the topology of pointwise convergence. {O}sipov \cite{MR4903474} studied the Fr\'echet--Urysohn property of the space $Q_p(X,\mathbb{R})$. The space of all real-valued quasicontinuous functions with the topology of uniform convergence on compacta was studied by {H}ol\'a et al.\@ \cite{MR4409699} and {H}ol\'a{} and {H}ol\'y \cite{MR4363318}.

%------------------------------------------------------------%
In this paper, we study cardinal invariants such as character ($\chi$), pseudocharacter ($\psi$), tightness ($t$), density tightness ($dt$), {L}indel\"of number ($l$), weak {L}indel\"of number ($wl$), hereditary {L}indel\"of number ($hl$), hereditary density ($hd$) and spread ($s$) related to the space $Q_p(X,Y)$. We obtain bounds for some of these cardinal invariants of $Q_p(X,Y)$. We also provide examples in most of these results where the inequalities are strict. Using the results in this paper, we determine character, tightness, hereditary {L}indel\"of number, hereditary density and spread of $Q_p(\mathbb{R},\mathbb{R})$.  More precisely, this paper is organized as follows.

Section~\ref{sec2} contains preliminary concepts. We define the notion of a quasicontinuous function and describe the topology of the space $Q_p(X,Y)$. We also introduce the subbasic and basic open subsets of $Q_p(X,Y)$ that will be used throughout the paper.

In Section~\ref{sec3}, we study character, pseudocharacter and tightness of $Q_p(X,Y)$.  We show that if $X$ is a regular space, then $Q_p(X,Y)$ separates disjoint closed and compact subsets of $X$. We give an example to show that such a separation is not necessarily true if $X$ is a Hausdorff space. Next, we prove that if $X$ is a Hausdorff space and $Y$ is a nontrivial $T_1$-space, then $\chi(Q_p(X,Y))=|X|\cdot \chi(Y)$. We show that for a regular space $X$ and a nontrivial $T_1$-space $Y$, $wl(X)\leq\psi(Q_p(X,Y))$ and $l(X^n)\leq t(Q_p(X,Y))$ for every $n\in\mathbb{N}$. For these inequalities, we provide examples in which the inequality is strict. We obtain that $t(Q_p(\mathbb{R},\mathbb{R}))=\mathfrak{c}$.

In Section~\ref{sec4}, we give bounds for {L}indel\"of number, hereditary {L}indel\"of number, hereditary density and spread of $Q_p(X,Y)$. We prove that for a regular space $X$ and a nontrivial $T_1$-space $Y$, $hd(X)\leq hl(Q_p(X,Y))$ and $hl(X)\leq hd(Q_p(X,Y))$. Furthermore, if $Y$ is Hausdorff, $dt(X^n)\leq l(Q_p(X,Y))$. Also, we prove that for a regular space $X$ and a nontrivial $T_1$-space $Y$, $c(X)\leq s(Q_p(X,Y))$, and $s(X^{n})\le s((Q_p(X,Y))^{n})$ for every $n\in\mathbb{N}$. Furthermore, if $Y$ is infinite, then $s(\Omega_n(X))\le s(Q_p(X,Y))$ for every $n\in\mathbb{N}$, where $\Omega_n(X)=\{(x_1,\dots,x_n)\in X^n: x_i\neq x_j \text{ whenever } i\neq j\}$. For each of these results, we provide examples in which the inequality is strict. We obtain that $hl(Q_p(\mathbb{R},\mathbb{R}))=hd(Q_p(\mathbb{R},\mathbb{R}))=s(Q_p(\mathbb{R},\mathbb{R}))=\mathfrak{c}$.

Thus, our main objective is to determine topological conditions on $X$ and $Y$ under which $Q_p(X,Y)$ has certain topological property. Although some results are analogue to  the results about $C_p(X,\mathbb{R})$, the proofs for $Q_p(X,Y)$ require properties specific to quasicontinuous mappings.

%%%%%%%%%%%%%%%%%%%%%%%%%%%%%%%%%%%%%%%%%%%%%%%%%%%%%%%%%%%%%%
%               Section: Preliminaries
%%%%%%%%%%%%%%%%%%%%%%%%%%%%%%%%%%%%%%%%%%%%%%%%%%%%%%%%%%%%%%
\section{Preliminaries}\label{sec2}

The space of real numbers with the usual topology is denoted by $\mathbb{R}$. The set of positive integers is denoted by $\mathbb{N}$. The cardinalities of $\mathbb{N}$ and $\mathbb{R}$ are denoted by $\aleph_0$ and $\mathfrak{c}$, respectively. Throughout this paper, $X,Y$ and $Z$ denote topological spaces unless otherwise stated. A nontrivial space is a topological space with more than one point. For a subset $A$ of $X$, $A^\mathrm{o}$ denotes the interior of $A$ in $X$ and $\overline{A}$ denotes the closure of $A$ in $X$. For a subspace $Y$ of $X$ and a subset $A$ of $Y$, the closure of $A$ relative to $Y$ is denoted by $\overline{A}_Y$.
 
 In \cite{Kempisty1932}, {K}empisty introduced quasicontinuity as a weaker form of continuity for real-valued functions. We use the following formulation of quasicontinuous function used by {N}eubrunn \cite{MR995972}.
%------------------------------------------------------------%
\begin{definition}\label{def: qcts}
A function $f:X\to Y$ is \emph{quasicontinuous at a point} $x\in X$ if for every open set $U$ containing $x$ and every open set $V$ containing $f(x)$, there exists a nonempty open set $G\subseteq U$ such that $f(G)\subseteq V$. If $f$ is quasicontinuous at every point of $X$, we say that $f$ is \emph{quasicontinuous}.
\end{definition}
%------------------------------------------------------------%
Note that every continuous function is quasicontinuous. Conversely, for $X=\mathbb{R}$ with the usual topology and $Y=\mathbb{R}$ with the {S}orgenfrey topology, the identity function from $X$ to $Y$ is quasicontinuous but nowhere continuous.

{L}evine \cite{MR166752} studied quasicontinuous functions under the name of semi-continuity using the terminology of semi-open sets. A subset $A$ of a space $X$ is said to be \emph{semi-open (or quasi-open)} \cite{MR995972}, if $A\subseteq \overline{A^\mathrm{o}}$. A function $f:X\to Y$ is \emph{semi-continuous} if for every open set $V$ in $Y$, $f^{-1}(V)$ is semi-open in $X$. A function $f:X\to Y$ is quasicontinuous if and only if it is semi-continuous.
%------------------------------------------------------------%
We use the following results from \cite{MR4491569} in most of our proofs.
\begin{lemma}\label{lem: sufficient_qcts}\cite[Lemma 4.1]{MR4491569}
Let $X$ and $Y$ be topological spaces and $f:X \rightarrow Y$ be a function such that for any $x\in X$, there exists an open set $G$ in $X$ such that $x\in \overline{G}$ and $f(y)=f(x)$ for all $y\in G$. Then $f$ is quasicontinuous.
\end{lemma}

\begin{lemma}\label{lem: qcts finite-point separation}\cite[Lemma 4.2]{MR4491569}
Let $X$ and $Y$ be topological spaces such that $X$ is Hausdorff. For given $x_1,\dots,x_n\in X$ and (not necessarily distinct) $y_1,\dots,y_n\in Y$, there exists a quasicontinuous function $f:X\rightarrow Y$ such that $f(x_i)=y_i$ for each $i\in\{1,\dots,n\}$.
\end{lemma}

\begin{lemma}\label{lem: qcts point-closed separation}\cite[Lemma 4.12]{MR4491569}
Let $X$ and $Y$ be spaces such that $X$ is regular. Then for any $x\in X$, any nonempty closed set $F\subseteq X$ such that $x\notin F$ and $y_1,y_2\in Y$, there exists a quasicontinuous function $f:X\to Y$ such that $f(x)=y_1$ and $f(F)=\{y_2\}$.
\end{lemma}
%------------------------------------------------------------%
Let $F(X,Y)$ be the set of all functions and $C(X,Y)$ be the set of all continuous functions from $X$ to $Y$. The function spaces $F(X,Y)$ and $C(X,Y)$ with the topology of pointwise convergence are denoted by $F_p(X,Y)$ and $C_p(X,Y)$, respectively. The space $C_p(X,\mathbb{R})$ has been widely studied in the literature (see, for example, \cite{MR1144519,MR1039321,MR953314,MR3024898}).

Let $X$ and $Y$ be topological spaces with topologies $\tau(X)$ and $\tau(Y)$, respectively. For $x\in X$ and $V\in\tau(Y)$, let $[x;V]=\{f\in F(X,Y):f(x)\in V\}$. Then the space $F_p(X,Y)$ has a subbasis $\mathcal{S}=\{[x;V]:x\in X,V\in\tau (Y)\}$.

Let $Q(X,Y)$ be the set of all quasicontinuous functions in $F(X,Y)$. The subspace $Q(X,Y)$ of the space $F_p(X,Y)$ is denoted by $Q_p(X,Y)$. Thus $C_p(X,Y)\subseteq Q_p(X,Y)\subseteq F_p(X,Y)$.
%------------------------------------------------------------%
For $x\in X$ and $V\in\tau (Y)$, denote
\[
[x;V]=\{f\in Q(X,Y):f(x)\in V\}.
\]
Then $\mathcal{S}'=\{[x;V]:x\in X,V\in\tau (Y)\}$ is a subbasis for the space $Q_p(X,Y)$. Observe that for $x\in X$ and $V_1,V_2\in\tau(Y)$, we have $[x;V_1]\cap [x;V_2]=[x;V_1\cap V_2]$. For $x_1,\dots,x_n\in X$ and $V_1,\dots,V_n\in\tau(Y)$, denote
\[[x_1,\dots,x_n;V_1,\dots,V_n]=\{f\in Q(X,Y):f(x_i)\in V_i,1\leq i\leq n\}.
\]
Clearly, the family
\[
\mathcal{B}=\{[x_1,\dots,x_n;V_1,\dots,V_n]:x_i\in X, V_i\in\tau (Y), 1\leq i\leq n,n\in\mathbb{N}\}
\]
is a basis for the space $Q_p(X,Y)$. Note that if $\mathcal{V}$ is a basis for $Y$, then the family
\[
\mathcal{B}'=\{[x_1,\dots,x_n;V_1,\dots,V_n]:x_i\in X,V_i\in\mathcal{V},1\leq i\leq n,n\in\mathbb{N}\}
\]
is also a basis for the space $Q_p(X,Y)$.

%%%%%%%%%%%%%%%%%%%%%%%%%%%%%%%%%%%%%%%%%%%%%%%%%%%%%%%%%%%%%%
\section{Cardinal bounds involving character and tightness of $Q_p(X,Y)$} \label{sec3}
%%%%%%%%%%%%%%%%%%%%%%%%%%%%%%%%%%%%%%%%%%%%%%%%%%%%%%%%%%%%%%

In this section, we give bounds for cardinal invariants of $Q_p(X,Y)$ which include character, pseudocharacter and tightness. We also discuss the results about $Q_p(X,Y)$ analogous to {A}rkhangel'skii-{P}ytkeev \cite{MR1144519} theorem. We will use the following lemma in the proof of Theorem~\ref{thm: l(X^n)<t(Q_p)} and Theorem~\ref{thm: dt(X^n)<l(Q_p)}.

%------------------------------------------------------------%
\begin{lemma}\label{lem: qcts compact-closed separation}
Let $X$ and $Y$ be spaces such that $X$ is regular. Let $K$ be a nonempty compact subset and $F$ be a nonempty closed subset of $X$ such that $K\cap F=\emptyset$ and $y_1,y_2\in Y$. Then there exists $f\in Q(X,Y)$ such that $f(K)=\{y_1\}$ and $f(F)=\{y_2\}$.
\end{lemma}

\begin{proof}
Since $X$ is regular and $F$ is a closed subset of $X$ such that $K\cap F=\emptyset$. For each $x\in K$ there exists an open subset $U_x$ of $X$ such that $x\in U_x\subseteq \overline{U_x}\subseteq X\setminus F$. Now, $\{U_x:x\in K\}$ is an open cover of $K$ and $K$ is compact, there exist $x_1,\ldots,x_n\in K$ such that $K\subseteq U_{x_1}\cup\dots\cup U_{x_n}$. Let $U=U_{x_1}\cup\dots\cup U_{x_n}$. Then $U$ is open, $K\subseteq U$, and $\overline U=\overline{U_{x_1}}\cup\dots\cup \overline{U_{x_n}}\subseteq X\setminus F$. Now, define a function $f:X\to Y$ such that
	\[ f(z)= \begin{cases}
		y_1 & \text{if $z\in \overline{U}$} \\
		y_2 & \ \text{if $z\in X\setminus\overline{U}$} \\
	\end{cases}. \]
By Lemma~\ref{lem: sufficient_qcts}, $f$ is a quasicontinuous function such that $f(K)=\{y_1\}$ and $f(F)=\{y_2\}$.
\end{proof}
%------------------------------------------------------------%
If the regularity assumption on $X$ in Lemmas~\ref{lem: qcts point-closed separation} and Theorem~\ref{lem: qcts compact-closed separation} is weakened to Hausdorff, the results need not hold, as the following example shows.

\begin{example}\label{exm: qcts compact-closed sep fail}
Consider $X=\mathbb{R}$ with the topology $\tau$ having a base $\mathcal{B}=\{U\setminus C:U$ is open in usual topology of $\mathbb{R}$ and $C\subseteq\mathbb{R}$ is countable$\}$. Let $Y=\mathbb{R}$ be the space with the usual topology. Since the usual topology is contained in $\tau$, it follows that $X$ is a Hausdorff space. Choose $p\in\mathbb{R}\setminus\mathbb Q$, and let $K=\{p\}$ and $F=\mathbb Q$. Then $K$ is compact and $F$ is closed in $X$ with $K\cap F=\emptyset$.

Suppose, to the contrary, that there exists $f\in Q(X,\mathbb{R})$ such that $f(p)=0$ and $f(F)=\{1\}$. Since $f$ is quasicontinuous at $p$, for the open neighborhood $X\setminus\mathbb Q$ of $p$ and for the open neighborhood $(-1/2,1/2)$ of $0$, there exists a nonempty open set $G\subseteq X\setminus\mathbb Q$ such that $f(G)\subseteq(-1/2,1/2)$. Choose a basic open set $B=I\setminus C\subseteq G$, where $I$ is an open interval in usual topology of $\mathbb{R}$ and $C$ is countable. Since $B\subseteq X\setminus\mathbb Q$, every rational point of $I$ belongs to $C$. Choose $q\in I\cap\mathbb Q$ and an open interval $J$ in usual topology of $\mathbb{R}$ such that $q\in J\subseteq I$. Then $f(q)=1$. Note that every nonempty open subset of $X$ contained in $J$ has nonempty intersection with $B$. In fact, if $H\subseteq J$ is nonempty open in $X$, then $H$ contains some basic open set $L\setminus C_1$, where $L\subseteq J$ is a nonempty open interval in usual topology of $\mathbb{R}$ and $C_1$ is countable; since $L$ is uncountable and $C\cup C_1$ is countable, there is $r\in L\setminus(C\cup C_1)$, and then $r\in H\cap B$.

Since $f$ is quasicontinuous at $q$ and $J$ is an open neighborhood of $q$ and the interval $(1/2,3/2)$ is an open neighborhood of $f(q)=1$. So there exists a nonempty open subset $H\subseteq J$ such that $f(H)\subseteq(1/2,3/2)$. But $H\cap B\neq\emptyset$ and $f(B)\subseteq(-1/2,1/2)$, a contradiction.
\end{example}

%------------------------------------------------------------%
The \emph{character} of a space $X$ is defined as $\chi(X)=\sup\{\chi(x,X):x\in X\}$, where $\chi(x,X)=\aleph_0+\min\{\vert\mathcal{B}_x\vert:\mathcal{B}_x$ is a base at $x\}$ \cite{MR3024898}.

The \emph{tightness} of a space $X$ is defined by $t(X)=\sup\{t(X,x):x\in X\}$, where $t(X,x)=\aleph_0+\min\{\kappa:$ for each $Y\subseteq X$ with $x\in\overline{Y}$ there is $A\subseteq Y$ with $|A|\leq\kappa$ and $x\in\overline{A}\}$ \cite{MR3024898}.
%------------------------------------------------------------%
\begin{theorem}\label{thm: char(Q_p)=|X| char(Y)}
	Let $X$ be a Hausdorff space and let $Y$ be a nontrivial $T_1$-space. Then $\chi(Q_p(X,Y))=|X|\cdot \chi(Y)$.
\end{theorem}

\begin{proof}
	First, we prove that $\chi(Q_p(X,Y))\leq |X|\cdot\chi(Y)$. Let $f\in Q(X,Y)$ be fixed. For each $x\in X$, let $\mathcal{B}_{f(x)}$ be a local base at $f(x)$ in $Y$ such that $|\mathcal{B}_{f(x)}|\leq\chi(Y)$. Then $\Gamma_f=\{[x_1,\dots,x_n;B_1,\dots,B_n]:x_i\in X, B_i\in\mathcal{B}_{f(x_i)},n\in\mathbb{N}\}$ is a local base at $f$ in $Q_p(X,Y)$. Hence, $\chi(f,Q_p(X,Y))\leq |\Gamma_f|\leq|X|\cdot\chi(Y)$. Since $f\in Q_p(X,Y)$ is arbitrary, we have $\chi(Q_p(X,Y))\leq |X|\cdot\chi(Y)$.
	
	To prove the converse, it suffices to show that $|X|\leq\chi(Q_p(X,Y))$ and $\chi(Y)\leq\chi(Q_p(X,Y))$. Fix $f\in Q(X,Y)$ and let $\mathcal{U}$ be a local base at $f$ in $Q_p(X,Y)$. For each $U\in\mathcal{U}$, choose a basic neighborhood $B_U=[x^U_1,\dots,x^U_n;V^U_1,\dots,V^U_n]$ of $f$ such that $f\in B_U\subseteq U$, and let $A=\{x^U_1,\dots,x^U_n:U\in\mathcal{U}\}$. Then $|A|\leq|\mathcal{U}|$. We claim that $A=X$. If possible, suppose that $x_0\in X\setminus A$. Choose $y_1\in Y$ with $y_1\neq f(x_0)$ and let $O=Y\setminus\{y_1\}$. Then $[x_0;O]=\{g\in Q_p(X,Y):g(x_0)\in O\}$ is an open neighborhood of $f$, so there exists $U\in\mathcal{U}$ such that $U\subseteq [x_0;O]$. Let $B_U=[x^U_1,\dots,x^U_n;V^U_1,\dots,V^U_n]$ and  $F=\{x^U_1,\dots,x^U_n\}$. Since $x_0\notin A$, we have $x_0\notin F$. By Lemma~\ref{lem: qcts finite-point separation}, there exists $h\in Q(X,Y)$ such that $h(x)=f(x)$ for every $x\in F$ and $h(x_0)=y_1$. Then $h\in B_U\subseteq U\subseteq [x_0;O]$, but $h(x_0)=y_1\notin O$, a contradiction. Thus $A=X$, and hence $|X|\leq|\mathcal{U}|$. Since $\mathcal{U}$ was arbitrary, $|X|\leq\chi(f,Q_p(X,Y))\leq\chi(Q_p(X,Y))$.
	
	To show $\chi(Y)\leq\chi(Q_p(X,Y))$, consider the injection map $i:Y\to Q_p(X,Y)$ such that $i(y)=c_y$ for each $y\in Y$, where $c_y$ is the constant function such that $c_y(x)=y$ for all $x\in X$. Note that $i$ is injective. Since for each $x\in X$ and each open $V$ in $Y$,  $i^{-1}([x;V])=V$ and $i(V)=[x;V]\cap i(Y)$. It is easy to see that the injection map $i$ is an embedding. Thus, $Y$ can be identified with the subspace $i(Y)$ of $Q_p(X,Y)$. Hence, $\chi(Y)\leq\chi(Q_p(X,Y))$.
\end{proof}
%------------------------------------------------------------%
\begin{corollary}
	Let $X$ be a Hausdorff space and $Y$ be a nontrivial first countable $T_1$-space. Then $Q_p(X,Y)$ is first countable if and only if $X$ is countable.
\end{corollary}
%------------------------------------------------------------%
\begin{corollary}\label{cor: t(Q)<X.Chi(Y)}
	Let $X$ be a Hausdorff space and let $Y$ be a nontrivial $T_1$-space. Then $t(Q_p(X,Y))\leq |X|\cdot \chi(Y)$.
\end{corollary}
\begin{proof}
	Since for any space $Z$, $t(Z)\leq \chi(Z)$ holds, the result follows from Theorem~\ref{thm: char(Q_p)=|X| char(Y)}.
\end{proof}
%------------------------------------------------------------%
In the next example using \cite[Example 4.1]{MR4363318}, we show that $t(Q_p(\mathbb{R},\mathbb{R}))\geq\mathfrak{c}$.
\begin{example}\label{exm: t(Q_p)>w}
Let $C$ be the standard Cantor set. For every finite set $F\subseteq C$, define $\phi_F:C\to\{0,1\}$ by
\[\phi_F(c)=
\begin{cases}
0 &\text{if }c\in F \\
1 & \text{if }c\in C\setminus F\\
\end{cases}
\]
Let us define $f_F:\mathbb{R}\to\mathbb{R}$ by
\[f_F(x)=
\begin{cases}
	\sin(1/d(x,C)) &\text{if }x\notin C \\
	\phi_F(x) & \text{if } x\in C\\
\end{cases}
\]
 where $d(x,C)=\inf\{|x-c|:c\in C\}$. Also define $h:\mathbb{R}\to\mathbb{R}$ by
 \[h(x)=
 \begin{cases}
 	\sin(1/d(x,C)) &\text{if }x\notin C \\
 	0 & \text{if } x\in C\\
 \end{cases}
 \]
 
By \cite[Example 4.1]{MR4363318}, each function $f_F$ and the function $h$ are quasicontinuous.
 
 Let $A=\{f_F:F\subseteq C$ and $F$ is finite$\}$. We first show that $h\in\overline{A}$. Let $U=[x_1,\dots,x_n;V_1,\dots,V_n]$ be a basic neighborhood of $h$ in $Q_p(\mathbb{R},\mathbb{R})$. Let $E=\{x_1,\dots,x_n\}$ and $F=E\cap C$. If $x\in E\setminus C$, then $f_F(x)=\sin(1/d(x,C))=h(x)$, while if $x\in E\cap C=F$, then $f_F(x)=0=h(x)$. Thus $f_F(x)=h(x)\in V_i$ for each $x\in E$, and hence $f_F\in U$. This proves that $h\in\overline{A}$.
 
 Now let $B\subseteq A$ with $|B|<\mathfrak{c}$. Thus $B=\{f_F:F\in\mathcal{F}\}$, where $\mathcal{F}$ is some family of finite subsets of $C$ with $|\mathcal{F}|<\mathfrak{c}$. Since every member of $\mathcal{F}$ is finite, $|\bigcup\mathcal{F}|<\mathfrak{c}$, and therefore there exists $c_B\in C\setminus\bigcup\mathcal{F}$. For every $F\in\mathcal{F}$, we have $c_B\notin F$, and hence $f_F(c_B)=1$, whereas $h(c_B)=0$. Hence, for the basic neighborhood $H=[c_B;(-1/2,1/2)]$ of $h$, we have $H\cap B=\emptyset$. Thus $h\notin\overline{B}$. It follows that $t(h,Q_p(\mathbb{R},\mathbb{R}))\geq\mathfrak{c}$ and hence $t(Q_p(\mathbb{R},\mathbb{R}))\geq\mathfrak{c}$.
\end{example}
%------------------------------------------------------------%
\begin{remark}
From Corollary~\ref{cor: t(Q)<X.Chi(Y)}, it follows that $t(Q_p(\mathbb{R},\mathbb{R}))\leq\mathfrak{c}$ and from Example~\ref{exm: t(Q_p)>w}, we have $t(Q_p(\mathbb{R},\mathbb{R}))\geq\mathfrak{c}$. Hence, $t(Q_p(\mathbb{R},\mathbb{R}))=\mathfrak{c}$.
\end{remark}
%------------------------------------------------------------%
The \emph{pseudocharacter of a point} $x$ in a $T_1$-space $X$ is $\psi(x,X)=\aleph_0+\min\{\vert\gamma\vert:\gamma$ is a family of open sets in $X$ such that $\cap\gamma=\{x\}\}$. The \emph{pseudocharacter} of a space $X$ is defined as $\psi(X)=\sup\{\psi(x,X):x\in X\} $\cite{MR3024898}.

The cardinal $l(X)=\aleph_0+\min\{\kappa$: every open cover has a subcover of cardinality $\leq \kappa\}$ is called the \emph{{L}indel\"of number} of the space $X$  \cite{MR3024898}. Note that a space is {L}indel\"of if $l(X)=\aleph_0$.

The \emph{weak {L}indel\"of number} or \emph{weak covering number} of a space $X$ is defined as $wl(X)=\aleph_0+\min\{\kappa:$ for every open cover $\mathcal{U}$ there exists $\mathcal{V}\subseteq \mathcal{U}$ such that $|\mathcal{V}|\leq \kappa$ and $\overline{\cup\mathcal{V}}=X \}$ \cite{MR0776620}. A space is \emph{weakly {L}indel\"of} if $wl(X)=\aleph_0$.
%------------------------------------------------------------%
\begin{theorem}\label{thm: wl(X)<psi(Q_p)}
	Let $X$ be a regular space and $Y$ be a nontrivial $T_1$-space. Then $wl(X)\leq \psi(Q_p(X,Y))$.
\end{theorem}
\begin{proof}
	Let $\psi(Q_p(X,Y))=\kappa$ and $\mathcal{U}$ be an open cover of $X$. We show that $\mathcal{U}$ has a subfamily $\mathcal{V}$ such that $|\mathcal{V}|\leq \kappa$ and $\overline{\cup \mathcal{V}}=X$.
	
	Let $y_0$ be a point in $Y$ and let $h_0$ be the constant function such that $h_0(x)=y_0$ for all $x\in X$. Since $\psi(Q_p(X,Y))=\kappa$, let $\gamma$ be a family of open sets in $Q_p(X,Y)$ containing $h_0$ such that $\cap \gamma=\{h_0\}$ and $|\gamma|\leq\psi(h_0,Q_p(X,Y))\leq\kappa$. Then $|\gamma|\leq \kappa$.
	
	For each $V\in \gamma$, there exists a basic open set $B_V=[x_1,\dots,x_n;U_1,\dots,U_n]$ such that $h_0\in B_V\subseteq V$. Note that $y_0\in U_i$ for each $i\in\{1,\dots,n\}$. Let $A_V=\{x_1,\dots,x_n\}$ for each $V\in\gamma$. For each $x\in A_V$, choose an open set $U_x\in\mathcal{U}$ containing $x$. Let $\mathcal{V}=\{U_x:x\in A_V,V\in\gamma\}$. Then $\mathcal{V}\subseteq\mathcal{U}$ and $|\mathcal{V}|\leq\kappa$. We claim that $\overline{\cup \mathcal{V}}=X$.
	
	Suppose that $z\in X\setminus \overline{\cup \mathcal{V}}$. Let $y_1$ be a point of $Y$ distinct from $y_0$. By Lemma~\ref{lem: qcts point-closed separation}, there exists a quasicontinuous function $h:X\to Y$ such that $h(z)=y_1$ and $h(\overline{\cup \mathcal{V}})=\{y_0\}$. Then $h\in B_V\subseteq V$ for all $V\in\gamma$ and thus $h\in\cap\gamma$. But $h\neq h_0$, which is a contradiction. Thus $\overline{\cup \mathcal{V}}=X$.
\end{proof}
%------------------------------------------------------------%

In Theorem~\ref{thm: wl(X)<psi(Q_p)} the equality, that is, $\psi(Q_p(X,Y))=wl(X)$, need not hold.
\begin{example} \label{exm411}
	Let $\mathbb{R}$ be the space of real numbers with the usual topology. Then $wl(\mathbb{R})\leq l(\mathbb{R})=\aleph_0$, while in \cite[Example 5.1.]{MR4116833}, it is shown that $\psi(Q_p(\mathbb{R},\mathbb{R}))=\mathfrak{c}$.
\end{example}

%------------------------------------------------------------%
Let $\tau$ be an infinite cardinal. For a {T}ychonoff space $X$, {A}rkhangel'skii \cite{MR1144519} proved that if $l(X^n) \le \tau$ for every $n \in \mathbb{N}$, then $t(C_p(X,\mathbb{R})) \le \tau$. Conversely, {P}ytkeev \cite{MR1144519} proved that if $t(C_p(X,\mathbb{R})) \le \tau$, then $l(X^n) \le \tau$ for every $n \in \mathbb{N}$. In particular, if $X$ is {T}ychonoff and compact, then $t(C_p(X,\mathbb{R}))=\aleph_0$.

In contrast to {A}rkhangel'skii's result, we show that the analogue of this result does not hold in the case of $Q_p(X,\mathbb{R})$. In Example~\ref{exm: t(Q_p)<l(X^n) fail}, we provide a normal, compact space $X$ for which $l(X^n)\leq\aleph_0$ for each $n\in \mathbb{N}$, but $t(Q_p(X,\mathbb{R}))>\aleph_0$. We prove an analogue of {P}ytkeev's result, for a regular space $X$ and a nontrivial $T_1$-space $Y$, if $t(Q_p(X,Y)) \le \tau$, then $l(X^n) \le \tau$ for every $n \in \mathbb{N}$.

%------------------------------------------------------------%
\begin{example} \label{exm: t(Q_p)<l(X^n) fail}
	Let $X=D\cup \{\infty\}$ be a one-point compactification of an uncountable discrete space $D$. Since $X$ is a compact Hausdorff space, using the {T}ychonoff theorem, it follows that $l(X^n)=\aleph_0$ for every $n\in \mathbb{N}$. We show that $t(Q_p(X,\mathbb{R}))>\aleph_0$.
	
	Let $D=M \cup N$, where $N$ is an infinite countable set and $M=D\setminus N$. For each finite set $F \subseteq M$, let us define $g_F:X\to \mathbb{R}$ such that
	\[ g_F(x)=
	\begin{cases}
		1&\text{if $x\in F$},\\
		0&\text{if $x\in X\setminus F$}.
	\end{cases}
	\]
	Also define $h:X \to \mathbb{R}$ by
	\[ h(x)=
	\begin{cases}
		1 & \text{if $x \in M$,}\\
		0 & \text{if $x \in N \cup \{\infty\}$}.\\
	\end{cases}
	\]
	Since every point of $D$ is isolated, $g_F$ is continuous at each point of $D$. At the point $\infty$, $g_F(\infty)=0$. For any neighborhood $V$ of $0$, $X\setminus F$ is an open neighborhood of $\infty$ such that $g_F(X\setminus F)=\{0\}$. Thus $g_F$ is continuous. Again, since every point of $D$ is isolated, $h$ is continuous at each point of $D$ and hence quasicontinuous at each point of $D$. At the point $\infty$, $h(\infty)=0$. Let $U$ be any neighborhood of $\infty$ in $X$ and  $V$ be neighborhood of $0$ in $\mathbb{R}$. Since $U$ contains all but finitely many points of $D$, $N \cap U \ne \emptyset$. So we can choose $n \in N \cap U$, so that $\{n\}$ is a nonempty open subset of $X$ contained in $U$ and $h(\{n\})=\{0\} \subseteq V$. Thus, $h$ is quasicontinuous.
	
	Let $A=\{g_F:F$ is finite and $F\subseteq M \}$. To show $h \in \overline{A}$ in $Q_p(X,\mathbb{R})$, let $W=[x_1;U_1] \cap \cdots \cap [x_m;U_m]$ be a basic neighborhood of $h$ in $Q_p(X,\mathbb{R})$. Let $K=\{x_1,\dots,x_m\}$ and $F_0=K\cap M$, so $g_{F_0} \in A$. For each $x_i\in F_0$, $g_{F_0}(x_i)=1=h(x_i)$ and for each $x_i\in N\cup \{\infty\}$, $g_{F_0}(x_i)=0=h(x_i)$.  Thus $g_{F_0} \in W$ and hence $h \in \overline{A}$.
	
	Suppose that $B \subseteq A$ is countable and thus $B=\{g_{F_n}:n \in \mathbb{N}\}$, where each $F_n$ is a finite subset of $M$. Let $S=\cup_{n \in \mathbb{N}} F_n$. Then $S$ is countable. Since $M$ is uncountable, we can choose $m_0 \in M \setminus S$. Then $U=[m_0;(\frac{1}{2},\frac{3}{2})]$ is a basic neighborhood of $h$ in $Q_p(X,\mathbb{R})$. Since for every $n \in \mathbb{N}$, $m_0 \notin F_n$, we have $g_{F_n}(m_0)=0$. Then $g_{F_n} \notin U$ and thus $B \cap U=\emptyset$. Hence $h \notin \overline{B}$. Therefore $h \in \overline{A}$, but $h$ does not belong to the closure of any countable subset of $A$. It follows that $t(Q_p(X,\mathbb{R}))>\aleph_0$.
\end{example}
%------------------------------------------------------------%
\begin{theorem} \label{thm: l(X^n)<t(Q_p)}
	Let $X$ be a regular space and $Y$ be a nontrivial $T_1$-space. Let $\tau$ be an infinite cardinal number. If $t(Q_p(X,Y)) \le \tau$, then $l(X^n) \le \tau$ for every $n \in \mathbb{N}$.
\end{theorem}

\begin{proof}
	Let $n \in \mathbb{N}$ be fixed and $\gamma$ be an open cover of $X^n$. Suppose that $y_1$ and $y_2$ are any two distinct points in $Y$ and $V$ is an open subset of $Y$ such that $y_1 \in V$ and $y_2 \notin V$. Let $c_{y_1}$ be the constant function from $X$ to $Y$ such that $c_{y_1}(x)=y_1$ for all $x\in X$. A finite family $\mu$ of open subsets of $X$ is $\gamma$-small if for every ordered $n$-tuple ($V_1,\dots,V_n) \in \mu^n$, there exists $G \in \gamma$ such that $V_1 \times \cdots \times V_n \subseteq G$. Let $\mathcal{E}$ be the family of all finite $\gamma$-small families of open subsets of $X$. For each $\mu \in \mathcal{E}$, denote $A_\mu=\{f \in Q(X,Y): f(x)=y_2 \text{ for all } x \in (X \setminus \cup \mu)\}$, and let $A=\cup_{\mu \in \mathcal{E}} A_\mu$. We claim that the constant function $c_{y_1}$ belongs to $\overline{A}$.
	
	Let $W=[x_1;U_1] \cap \cdots \cap [x_m;U_m]$ be any basic neighborhood of $c_{y_1}$ in $Q_p(X,Y)$. Let $K=\{x_1,\dots,x_m\}$. Since $K^n$ is finite, for each $\xi=(z_1,\dots,z_n) \in K^n$, we can choose $G_\xi \in \gamma$ such that $\xi \in G_\xi$. Then there exist open sets $V_1^\xi,\dots,V_n^\xi \subseteq X$ such that $z_i \in V_i^\xi$ for each $i$ and $V_1^\xi \times \cdots \times V_n^\xi \subseteq G_\xi$. Let $\theta_K$ be the finite family of all such sets $V_i^\xi$. For each $x \in K$, let $W_x=\cap \{O \in \theta_K : x \in O\}$, and let $\mu_K=\{W_x : x \in K\}$. Since $K$ is finite, each $W_x$ is open and $K \subseteq \cup \mu_K$. Also $\mu_K$ is $\gamma$-small, because if $W_{x_1'},\dots,W_{x_n'} \in \mu_K$, then for $\eta=(x_1',\dots,x_n') \in K^n$ there exist $O_1,\dots,O_n \in \theta_K$ and $G \in \gamma$ such that $x_i' \in O_i$ for each $i$ and $O_1 \times \cdots \times O_n \subseteq G$, while $W_{x_i'} \subseteq O_i$ for each $i$, hence $W_{x_1'} \times \cdots \times W_{x_n'} \subseteq G$. Thus $\mu_K \in \mathcal{E}$.
	
	Let $F=(X \setminus \cup \mu_K)$. If $F=\emptyset$, then $c_{y_1} \in A_{\mu_K} \cap W$. Assume that $F \ne \emptyset$. Then $F$ is closed and $x_i \notin F$ for each $i$, so by Lemma~\ref{lem: qcts compact-closed separation} there exists a quasicontinuous function $g:X \to Y$ such that $g(x_i)=y_1$ for each $i \in \{1,\dots,m\}$ and $g(F)=\{y_2\}$. Hence $g \in A_{\mu_K}\cap W$. Thus $c_{y_1} \in \overline{A}$.
	
	Since $t(Q_p(X,Y)) \le \tau$, there exists $B \subseteq A$ such that $|B| \le \tau$ and $c_{y_1} \in \overline{B}$. For each $f \in B$, choose $\mu(f) \in \mathcal{E}$ with $f \in A_{\mu(f)}$ and let $\mathcal{E}_0=\{\mu(f): f \in B\}$. Then $|\mathcal{E}_0| \le \tau$ and $B\subseteq \cup_{\mu\in\mathcal{E}_0}A_\mu$. For each $\mu \in \mathcal{E}_0$ and each $\xi=(V_1,\dots,V_n) \in \mu^n$, choose $G_\xi \in \gamma$ such that $V_1 \times \cdots \times V_n \subseteq G_\xi$. Let $\gamma_\mu=\{G_\xi : \xi \in \mu^n\}$ and $\widetilde{\gamma}=\cup_{\mu \in \mathcal{E}_0} \gamma_\mu$. Since $|\mathcal{E}_0| \le \tau$ and $\gamma_\mu$ is finite for every $\mu\in \mathcal{E}_0$, we have $|\widetilde{\gamma}| \le \tau$. 
	
	We show that $\widetilde{\gamma}$ covers $X^n$. Let $(x_1,\dots,x_n) \in X^n$ be arbitrary. Consider the basic neighborhood $O=[x_1;V] \cap \cdots \cap [x_n;V]$ of $c_{y_1}$. Since $c_{y_1} \in \overline{B}$, there exists $f \in B \cap O$. Let $\mu_0 \in \mathcal{E}_0$ be such that $f \in A_{\mu_0}$. Then $f(x_i) \in V$, for each $i\in\{1,\dots,n\}$ and $f(x)=y_2$, for all $x \in (X \setminus \cup \mu_0)$. Since $y_2 \notin V$, it follows that $x_i \in \cup \mu_0$ for each $i$. Thus, we can choose $V_i \in \mu_0$ such that $x_i \in V_i$ for each $i$. Since $\mu_0$ is $\gamma$-small, there exists $G \in \gamma_{\mu_0} \subseteq \widetilde{\gamma}$ such that $V_1 \times \cdots \times V_n \subseteq G$. Hence, $(x_1,\dots,x_n) \in G$. Thus $\widetilde{\gamma}$ is a subcover of $\gamma$ of cardinality less than or equal to $\tau$. This implies $l(X^n) \le \tau$. Since $n$ was arbitrary, $l(X^n) \le \tau$ for every $n \in \mathbb{N}$.
\end{proof}
%------------------------------------------------------------%
\begin{corollary} \label{cor: tightness Q_p_lindelof X}
Let $X$ be a regular space and $Y$ be a nontrivial $T_1$-space. If $t(Q_p(X,Y))=\aleph_0$ then $X$ is a {L}indel\"of space.
\end{corollary}
%------------------------------------------------------------%
The equality in Theorem~\ref{thm: l(X^n)<t(Q_p)} need not hold.
\begin{example}\label{exm: t(Q_p(X)>l(X)}
Let $X$ be the {S}orgenfrey line. Since $X$ is {L}indel\"of, $l(X)=\aleph_0$. We claim that $t(Q_p(X,\mathbb{R}))>\aleph_0$. Note that from \cite[Problem 149]{MR3024898}, we have $l(X^n)\leq t(C_p(X,\mathbb{R}))$ for all $n\in\mathbb{N}$ and using \cite[Problem 159]{MR3024898}, we have $t(C_p(X,\mathbb{R}))\leq t(Q_p(X,\mathbb{R}))$. Since $X^2$ is not {L}indel\"of, it follows that $t(Q_p(X,\mathbb{R}))\geq t(C_p(X,\mathbb{R}))\geq l(X^2)>\aleph_0$.
\end{example}
%------------------------------------------------------------%
\begin{remark}\label{rem: Tychonoff l(X)<t(Q_p)}
When $X$ is a {T}ychonoff space, $l(X^n)\leq t(C_p(X,\mathbb{R}))$ \cite[Problem 149]{MR3024898}, and $t(C_p(X,\mathbb{R}))\leq t(Q_p(X,\mathbb{R}))$ \cite[Problem 159]{MR3024898}. Thus it is immediate that for a {T}ychonoff space $X$, $l(X^n)\leq t(Q_p(X,\mathbb{R}))$ holds. Also from the Corollary~\ref{cor: tightness Q_p_lindelof X}, it follows that a regular $T_1$-space $X$ with countable tightness of $Q_p(X,\mathbb{R})$ is a {L}indel\"of space and hence a {T}ychonoff space.
\end{remark}

%%%%%%%%%%%%%%%%%%%%%%%%%%%%%%%%%%%%%%%%%%%%%%%%%%%%%%%%%%%%%%
% Section: Cardinal bounds involving {L}indel\"of number and spread of $Q_p(X,Y)$
%%%%%%%%%%%%%%%%%%%%%%%%%%%%%%%%%%%%%%%%%%%%%%%%%%%%%%%%%%%%%%
\section{Cardinal bounds involving {L}indel\"of number and spread of $Q_p(X,Y)$} \label{sec4}
In this section, we obtain results which give bounds for {L}indel\"of number, hereditary {L}indel\"of number, hereditary density and spread of $Q_p(X,Y)$. Using these results, we obtain that $hl(Q_p(\mathbb{R},\mathbb{R}))=hd(Q_p(\mathbb{R},\mathbb{R}))=s(Q_p(\mathbb{R},\mathbb{R}))=\mathfrak{c}$.

%------------------------------------------------------------%
The \emph{hereditary {L}indel\"of number} of a space $X$ is defined by $hl(X)=\sup\{l(Y):Y\subseteq X\}$, where $l(Y)$ is the {L}indel\"of number of the subspace $Y$ of $X$ \cite{MR3024898}. A space $X$ is called \emph{hereditarily $\kappa$-{L}indel\"of} if $hl(X)\leq\kappa$. A space $X$ is \emph{hereditarily {L}indel\"of} if $hl(X)=\aleph_0$ \cite{MR1039321}.

The density of a space $X$ is defined as $d(X) =\aleph_0+ \min\{|D|: D \subseteq X \text{ and } \overline{D} = X\}$.

The \emph{hereditary density} of a space $X$ is defined by $hd(X)=\sup\{d(Y):Y\subseteq X\}$, where $d(Y)$ is the density of the subspace $Y$ of $X$ \cite{MR3024898}. A space $X$ is called \emph{hereditarily $\kappa$-separable} if $hd(X)\leq\kappa$. A space $X$ is \emph{hereditarily separable} if $hd(X)=\aleph_0$ \cite{MR1039321}.
%------------------------------------------------------------%
\begin{theorem}\label{thm: hl(X)<hd(Q_p)}
Let $X$ be a regular space and $Y$ be a nontrivial $T_1$-space. Let $\kappa$ be an infinite cardinal. If $hd(Q_p(X,Y))\leq \kappa$ then $hl(X)\leq \kappa$.
\end{theorem}

\begin{proof}
Suppose that $hl(X)>\kappa$, that is, $X$ is not hereditarily $\kappa$-{L}indel\"of. Then there exists a subspace $Z$ of  $X$ such that $l(Z)>\kappa$. Let $\mathcal{U}$ be an open cover of $X$ having no subfamily of cardinality less than or equal to $\kappa$ which cover $Z$. For each $z\in Z$, choose $U_z\in\mathcal{U}$ such that $z\in U_z$. Let $y_0$ and $y_1$ be any two distinct points in $Y$. Since $X$ is regular and $Y$ is a nontrivial $T_1$-space, by Lemma~\ref{lem: qcts point-closed separation} there exists a quasicontinuous function $f_z:X\to Y$ such that $f_z(X\setminus U_z)=\{y_0\}$ and $f_z(z)=y_1$. Consider $F=\{f_z:z\in Z\}$. We claim that for $F\subseteq Q_p(X,Y)$, $d(F) > \kappa$.

Let $D$ be a dense subset of $F$ with $|D|\leq \kappa$. Let $V$ be an open neighborhood of $y_1$ in $Y$ such that $y_0\notin V$. For each $z\in Z$, we have $([z;V]\cap F)\cap D\neq\emptyset$. Thus there exists $f_{z'}\in D$ such that $f_{z'}(z)\in V$. Since $f_{z'}(X\setminus U_{z'})=\{y_0\}$, $z\in U_{z'}$. Therefore $\{U_{z'}: f_{z'} \in D\}$ is a subfamily of $\mathcal{U}$ of cardinality $\leq \kappa$ which covers $Z$, which is a contradiction to the fact that $\mathcal{U}$ has no subfamily of $\mathcal{U}$ of cardinality $\leq \kappa$ which covers $Z$.
Thus $d(F) > \kappa$. Since $F \subseteq Q_p(X,Y)$, this implies $hd(Q_p(X,Y)) > \kappa$, which contradicts the hypothesis. Thus $hl(X)\leq \kappa$.
\end{proof}
%------------------------------------------------------------%
\begin{corollary}\label{cor: hs_Q_p implies X hl}
Let $X$ be a regular space and $Y$ be a nontrivial $T_1$-space. If $Q_p(X,Y)$ is hereditarily separable then $X$ is hereditarily {L}indel\"of.
\end{corollary}
%------------------------------------------------------------%
\begin{theorem}\label{thm: hd(X)<hl(Q_p)}
	Let $X$ be a regular space and $Y$ be a nontrivial $T_1$-space. Let $\kappa$ be an infinite cardinal. If $hl(Q_p(X,Y)) \leq \kappa$ then $hd(X) \leq \kappa$.
\end{theorem}

\begin{proof}
Suppose that $hd(X)>\kappa$, that is, $X$ is not hereditarily $\kappa$-separable. Let $Z$ be a subspace of $X$ such that $d(Z) > \kappa$. Let $A$ be a subset of $Z$ with $|A| \leq \kappa$. Since $A$ cannot be dense in $Z$, there exists $z_A\in Z$ such that $z_A\notin \overline{A_Z}$. Then $z_A\notin \overline{A}$. Let $y_0$ and $y_1$ be any two distinct points in $Y$. Since $X$ is regular, for each subset $A$ of $Z$ with $|A| \leq \kappa$, by Lemma~\ref{lem: qcts point-closed separation}, there exists a quasicontinuous function $f_A:X\to Y$ such that $f_A(\overline{A})=\{y_0\}$ and $f_A(z_A)=y_1$. Let $W=\{f_A:A\subseteq Z, |A| \leq \kappa\}$. We claim that $W$ is not a $\kappa$-{L}indel\"of subspace of the space $Q_p(X,Y)$.
	
Let $V$ be an open neighborhood of $y_1$ in $Y$ such that $y_0\notin V$. For each $f_A\in W$, let $U_A=[z_A;V]\cap W$ which is open in $W$ and contains $f_A$. Let $\mathcal{U}=\{U_A:f_A\in W\}$. Then $\mathcal{U}$ is an open cover of $W$. We will show that $\mathcal{U}$ has no subcover of cardinality less than or equal to $\kappa$ for $W$.
	
Let $\{U_{A_\alpha}:\alpha \in I\}$ be a subcollection of $\mathcal{U}$ where $|I| \leq \kappa$. Let $B=\bigcup_{\alpha \in I} (A_\alpha \cup \{z_{A_\alpha}\})$. Since $|A_\alpha| \leq \kappa$ for each $\alpha$ and $|I| \leq \kappa$, we have $|B| \leq \kappa$. Thus $B$ is a subset of $Z$ with $|B| \leq \kappa$. Then there exists $z_B\in Z$ such that $z_B\notin \overline{B}$. Thus again by Lemma~\ref{lem: qcts point-closed separation}, there is a quasicontinuous function $f_B:X\to Y$ such that $f_B(\overline{B})=\{y_0\}$ and $f_B(z_B)=y_1$. Note that for each $\alpha \in I$, $z_{A_\alpha}\in B$, so that $f_B(z_{A_\alpha})=y_0$. This shows that $f_B\notin U_{A_\alpha}$ for all $\alpha \in I$. Since $f_B\in W$, this contradicts the fact that $\{U_{A_\alpha}:\alpha \in I\}$ is a subcover of $W$. Thus $hd(X) \leq \kappa$.
\end{proof}
%------------------------------------------------------------%
\begin{corollary}\label{cor: hl_Q_p implies X hs}
Let $X$ be a regular space and $Y$ be a nontrivial $T_1$-space. If $Q_p(X,Y)$ is hereditarily {L}indel\"of then $X$ is hereditarily separable.
\end{corollary}
%------------------------------------------------------------%

The converses of Theorem~\ref{thm: hl(X)<hd(Q_p)} and Theorem~\ref{thm: hd(X)<hl(Q_p)} need not be true.
\begin{example}\label{exm: discrete D in Q_R}
Let $\mathbb{R}$ be the space of real numbers with the usual topology and $Y$ be any nontrivial $T_1$-space. Since $\mathbb{R}$ is second countable, every subspace of $\mathbb{R}$ is second countable and hence {L}indel\"of and separable. Thus, $hl(\mathbb{R})=\aleph_0$ and $hd(\mathbb{R})=\aleph_0$. We claim that $hd(Q_p(\mathbb{R},Y))>\aleph_0$ and also $hl(Q_p(\mathbb{R},Y))>\aleph_0$.
	
Let $y_0$ and $y_1$ be any two distinct points in $Y$ and let $V$ be an open subset of $Y$ such that $y_1\in V$ and $y_0\notin V$. For each $a\in\mathbb{R}$, define $f_a:\mathbb{R}\to Y$ such that
	$$f_a(t)=
	\begin{cases}
		y_1 & \text{if } t\in [a,a+1]\\
		y_0 & \text{if } t\notin [a,a+1].
	\end{cases}
	$$
By Lemma~\ref{lem: sufficient_qcts}, each $f_a$ is quasicontinuous. Let $D=\{f_a:a\in\mathbb{R}\}$. For each $a\in\mathbb{R}$, let $\mathcal{N}_a=\{f\in Q_p(\mathbb{R},Y):f(a)\in V\text{ and }f(a+1)\in V\}$. Note that $\mathcal{N}_a=[a;V]\cap[a+1;V]$ and hence it is open in $Q_p(\mathbb{R},Y)$ and contains $f_a$. If $b\neq a$ and $f_b\in\mathcal{N}_a$, then $a\in [b,b+1]$ and $a+1\in [b,b+1]$, a contradiction. Hence $\mathcal{N}_a\cap D=\{f_a\}$ for every $a\in\mathbb{R}$, and therefore $D$ is a discrete subspace of $Q_p(\mathbb{R},Y)$ of cardinality $\mathfrak{c}$. Thus $hd(Q_p(\mathbb{R},Y))\geq \mathfrak{c}>\aleph_0=hl(\mathbb{R})$. Also, note that a discrete space of cardinality $\mathfrak{c}$ has {L}indel\"of number $\mathfrak{c}$. Thus, $hl(Q_p(\mathbb{R},Y))\geq \mathfrak{c}>\aleph_0=hd(\mathbb{R})$.
\end{example}
%------------------------------------------------------------%
\begin{remark}
Since for any space $Z$, $hl(Z)\leq w(Z)$ and $hd(Z)\leq w(Z)$ holds. Using Example~\ref{exm: discrete D in Q_R} and from \cite{MR4116833} using $w(Q_p(\mathbb{R},\mathbb{R}))=\mathfrak{c}$, we conclude that $hl(Q_p(\mathbb{R},\mathbb{R}))=hd(Q_p(\mathbb{R},\mathbb{R}))=\mathfrak{c}$.
\end{remark}
%------------------------------------------------------------%
\begin{remark}
For a {T}ychonoff space $X$, it is known that $hd(X)\leq hl(C_p(X,\mathbb{R}))$ and $hl(X)\leq hd(C_p(X,\mathbb{R}))$ \cite[Corollary 4.8.4]{MR953314}. Also, $hl(C_p(X,\mathbb{R}))\leq hl(Q_p(X,\mathbb{R}))$ and $hd(C_p(X,\mathbb{R}))\leq hd(Q_p(X,\mathbb{R}))$ hold in general. In this case results of Theorem~\ref{thm: hl(X)<hd(Q_p)} and Theorem~\ref{thm: hd(X)<hl(Q_p)}, that is, $hd(X)\leq hl(Q_p(X,\mathbb{R}))$ and $hl(X)\leq hd(Q_p(X,\mathbb{R}))$ follow directly from these results.
\end{remark}
%------------------------------------------------------------%
\begin{theorem}\label{thm: l(Q_p)<|X|}
Let $X$ be an infinite Hausdorff space and $Y$ be a nontrivial second countable space then $l(Q_p(X,Y))\leq |X|$.
\end{theorem}

\begin{proof}
By \cite[Theorem 4.7]{MR4491569}, we have $w(Q_p(X,Y))=|X|$. For any space $Z$, the result $l(Z)\leq w(Z)$ holds. Hence $l(Q_p(X,Y))\leq |X|$ holds.
\end{proof}
%------------------------------------------------------------%
\begin{corollary}\label{cor413}
Let $X$ be an infinite countable Hausdorff space and $Y$ be a nontrivial second countable space then $Q_p(X,Y)$ is a {L}indel\"of space.
\end{corollary}
%------------------------------------------------------------%
The equality in Theorem~\ref{thm: l(Q_p)<|X|}, that is, $l(Q_p(X,Y))= |X|$, need not hold.

\begin{example}\label{exm414}
Let $X=D\oplus S$ be the topological sum, where $S=\{0\}\cup\{1/m:m\in\mathbb{N}\}\subseteq\mathbb{R}$ has the usual subspace topology and let $D$ be an uncountable discrete space. Let $Y=\{0,1\}$ be a discrete space. Then $X$ is regular $T_1$-space and $|X|=|D|$. Since $X$ is a topological sum, it is easy to see that a map $f:X\to Y$ is quasicontinuous if and only if $f|_D$ and $f|_S$ are quasicontinuous. Hence $Q_p(X,Y)\cong Q_p(D,Y)\times Q_p(S,Y)$. Since $D$ is discrete, $Q_p(D,Y)=Y^D$ and it is compact by {T}ychonoff's theorem. Also $S$ is countable and  $Q_p(S,Y)\subseteq Y^S$ and thus $Q_p(S,Y)$ is second countable and hence it is {L}indel\"of. Since the product of a compact space and a {L}indel\"of space is {L}indel\"of, $Q_p(X,Y)$ is {L}indel\"of. Thus $l(Q_p(X,Y))=\aleph_0<|X|$.
\end{example}
%------------------------------------------------------------%
Recall that \cite{MR1919314} a subset $Y$ of a space $X$ is called \emph{$\kappa$-dense} if $[Y]_{\kappa}=X$ where
$[Y]_{\kappa}=\bigcup\{\overline{B} : B\subseteq Y \text{ and } |B|\le \kappa\}$. The set $[Y]_{\kappa}$ is called the \emph{$\kappa$-closure} of $Y$ in $X$. The cardinal $dt(X)=\aleph_0+\min\{\kappa:\text{every dense subset of $X$ is $\kappa$-dense in $X$}\}$ is called the \emph{density tightness} or \emph{$d$-tightness} of the space $X$.

The density tightness of a space $X$ can also be described as $dt(X)=\sup\{dt(x,X):x\in X\}$, where $dt(x,X)=\aleph_0+\min\{\kappa:$ for each dense subset $D$ of $X$, $\exists$ subset $B$ of $D$ such that $|B|\leq \kappa$ and $x\in\overline{B}\}$. It follows that $dt(X)\leq t(X)$, for any space $X$.

{A}sanov \cite{Asanov1} proved that $t(X^n)\leq l(C_p(X,\mathbb{R}))$ for a {T}ychonoff space $X$. In the next result, we compare the {L}indel\"of number of the space $Q_p(X,Y)$ and the density tightness of the space $X^n$.
%------------------------------------------------------------%
\begin{theorem}\label{thm: dt(X^n)<l(Q_p)}
	Let $X$ be a regular $T_1$-space and $Y$ be a nontrivial Hausdorff space. Then for every $n\in\mathbb{N}$, $dt(X^n)\leq l(Q_p(X,Y))$.
\end{theorem}

\begin{proof}
	Let $l(Q_p(X,Y))=\kappa$ and $n\in \mathbb{N}$ be fixed. Let $D$ be a dense subset in $X^n$ and $x=(x_1,\dots,x_n)\in X^n$. To show $dt(X^n)\leq \kappa$, we show that there is $B\subseteq D$ such that $|B|\leq \kappa$ and $x\in \overline{B}$.
	
	Since $X$ is Hausdorff, we can choose open subsets $U_1,\dots,U_n$ of $X$ such that for each $i\in\{1,\dots,n\}$ and $j\in\{1,\dots,n\}$,
	\begin{equation*}
		\begin{cases}
			\text{$U_i=U_j$}, & \text{if $x_i=x_j$}, \\
			\text{$U_i\cap U_j=\emptyset$}, & \text{if $x_i\neq x_j$},\\
			\text{$x_i\in U_i$}, & \text{for each $i\in\{1,\dots,n\}$}.\\
		\end{cases}\tag{$*$}
	\end{equation*}
	Then $U=U_1\times\dots\times U_n$ is an open set in $X^n$ containing $x$. Since $D\cap U$ is dense in $U$ and $x \in \overline{D \cap U}$. For any $B\subseteq D\cap U$, then $B\subseteq D$, thus it suffices to find $B\subseteq D\cap U$ with $x\in \overline{B}$.
	
	Let $y_0$ and $y_1$ be any two distinct points in $Y$. Let $\mathcal{A}= \{f \in Q_p(X,Y): f(x_i) = y_0, \text{ for each }i\in \{1,\dots,n\}\}$. Since the evaluation function at each $a\in X$, $e_a:Q_p(X,Y)\to Y$ such that $e_a(f)=f(a)$ is continuous \cite[Lemma 5.5]{MR4491569} and $Y$ is a $T_1$-space, thus $\mathcal{A}=\cap_{i=1}^n e^{-1}_{x_i}(\{y_0\})$ is a closed subset of the space $Q_p(X,Y)$. Then $l(\mathcal{A}) \le \kappa$. Let $\mathcal{V}_z = \{g \in Q_p(X,Y): g(z_i)\neq y_1,\text{ for each }i\in \{1,\dots,n\}\}$ for all $z = (z_1, \dots, z_n) \in D$. Take arbitrary $f \in \mathcal{A}$. Since $x \in \overline{D}$, $f$ is quasicontinuous, $Y$ is nontrivial Hausdorff and $f(x_i) = y_0$ for all $i\in \{1,\dots,n\}$, there is a point $z= (z_1, \dots, z_n) \in D$ such that $f(z_i)\neq y_1$ for all $i\in \{1,\dots,n\}$. Thus $\mathcal{A}\subseteq\cup\{\mathcal{V}_z: z \in D\}$. Since $l(\mathcal{A}) \le \kappa$, there is a $B \subseteq D$ such that $|B| \le \kappa$ and $\mathcal{A} \subseteq \cup\{\mathcal{V}_z: z \in B\}$. We show that $x \in \overline{B}$.
	
	Suppose $x \notin \overline{B}$. Since $X$ is regular, by Lemma~\ref{lem: qcts compact-closed separation}, there exists $h \in Q_p(X,Y)$ such that $h \in \mathcal{A}$ and $h(X \setminus \cup_{i=1}^n U'_i) = \{y_1\}$, where $U'_i$ are open sets in $X$ satisfying $(*)$ and $U'_i \subseteq U_i$, $(U'_1 \times \dots \times U'_n) \cap B = \emptyset$. Since $h \in \mathcal{V}_{z'}$ for some $z' \in B$, we have $h(z'_i)\neq  y_1$, for $i\in\{1,\dots,n\}$, where $z'_i$ is the $i$th coordinate of the point $z'$. Since $z' \in B\subseteq (D\cap U) \subseteq U$, for each $i\in\{1,\dots,n\}$ we have $z'_i \in U_i$. Thus $z'_i \in U'_i$, since $h(z'_i) \neq y_1$ and $U_i$ is disjoint from $U'_j$ if $x_i \ne x_j$. Therefore $z' \in (U'_1 \times \dots \times U'_n) \cap B$; a contradiction.
\end{proof}
%------------------------------------------------------------%
The equality in Theorem~\ref{thm: dt(X^n)<l(Q_p)}, that is, $l(Q_p(X,Y))= dt(X^n)$ for every $n\in\mathbb{N}$ need not hold.

\begin{example}\label{exm: dt(X^n)=l(Q_p) fail}
	Let $\kappa$ be an infinite cardinal number. Let $X=S=\{0\}\cup\{1/m:m\in\mathbb{N}\}$ be the subspace of $\mathbb{R}$ with the usual topology and let $Y$ be a discrete space with $|Y|>\kappa$. Note that $X$ is a compact metrizable space and hence it is regular. Every finite power $X^n$ is metrizable and hence it is first countable for each $n\in\mathbb{N}$. Using $dt(X^n)\leq t(X^n)$ and \cite[1.7.13]{MR1039321}, we have $dt(X^n)\leq t(X^n)=\aleph_0\leq \kappa$. Let $x_0$ be an isolated point in $X$. We define $e_{x_0}:Q_p(X,Y)\to Y$ such that $e_{x_0}(f)=f(x_0)$. Then $e_{x_0}$ is continuous and onto. Thus, $l(Q_p(X,Y))\geq l(Y)>\kappa$.
\end{example}
%------------------------------------------------------------%
The \emph{cellularity} or \emph{Souslin number} of a space $X$ is defined by $c(X)=\aleph_0+\sup\{\vert\mathcal{U}\vert:\mathcal{U}$ is a family of pairwise disjoint nonempty open subsets of $X\}$ \cite{MR3024898}. A space $X$ is said to have \emph{Souslin property} if $c(X)=\aleph_0$.

The \emph{spread} of a space $X$ is  defined as $s(X)=\aleph_0+\sup\{|D|:D\subseteq X \text{ is discrete}\}$ \cite{MR3024898}.

The weight of a space $X$ is defined by $w(X)=\aleph_0+\min\{\vert\mathcal{B}\vert:\mathcal{B}$ is a base for $X\}$.
%------------------------------------------------------------%
\begin{theorem} \label{thm: c(X)< s(Q_p)}
Let $X$ be a topological space and let $Y$ be a nontrivial $T_1$-space. Then $c(X)\leq s(Q_{p}(X,Y))$.
\end{theorem}

\begin{proof}
Let $\mathcal{U}$ be a pairwise disjoint family of nonempty open subsets of $X$. For each $U\in\mathcal{U}$, choose a point $x_{U}\in U$. Let $y_0$ and $y_1$ be any two distinct points in $Y$ and let $V$ be an open subset of $Y$ such that $y_0\in V$ and $y_1\notin V$. Let us define $f_{U}:X\rightarrow Y$ by
\[ f_{U}(x)=
	\begin{cases}
		y_0, & x\in\overline{U},\\
		y_1, & x\in X\setminus\overline{U}.
	\end{cases}
\]
By Lemma~\ref{lem: sufficient_qcts}, $f_{U}$ is quasicontinuous. Let $\mathcal{A}=\{f_{U}:U\in\mathcal{U}\}$. We claim that $\mathcal{A}$ is a discrete subset of $Q_p(X,Y)$.

For each $U\in\mathcal{U}$, consider the open set
$N_{U}=\{f\in Q_{p}(X,Y):f(x_{U})\in V\}$, which contains $f_{U}$. For any $W\in\mathcal{U}$ with $W\neq U$, $U$ and $W$ are disjoint open subsets of $X$. Therefore $x_{U}\notin\overline{W}$ and $f_{W}(x_{U})=y_1\notin V$. Thus $f_{W}\notin N_{U}$, and hence $N_{U}\cap\mathcal{A}=\{f_{U}\}$, which shows $\mathcal{A}$ is a discrete subspace of $Q_{p}(X,Y)$. Note that $|\mathcal{U}|=|\mathcal{A}|\leq s(Q_{p}(X,Y))$. Now taking the supremum over all pairwise disjoint families of nonempty open subsets of $X$, we have $c(X)\leq s(Q_{p}(X,Y))$.
\end{proof}

%------------------------------------------------------------%
\begin{corollary}\label{cor: c(X)< hl(Q_p)}
Let $X$ be a space and $Y$ be any nontrivial $T_1$-space. Then $c(X)\leq hl(Q_p(X,Y))$.
\end{corollary}

\begin{proof}
Since for any space $Z$, by definition $s(Z)\leq hl(Z)$ holds. By Theorem \ref*{thm: c(X)< s(Q_p)} the result follows .
\end{proof}
%------------------------------------------------------------%
The equality in Theorem~\ref{thm: c(X)< s(Q_p)}, that is, $s(Q_p(X,Y))= c(X)$, need not hold.

\begin{example}\label{exm: c(X)= s(Q_p) fail}
Let $\mathbb{R}$ be the space of real numbers with the usual topology. Then $c(\mathbb{R})=\aleph_0$. By using the set $D$ in Example~\ref{exm: discrete D in Q_R}, it is easy to see that $s(Q_p(\mathbb{R},\mathbb{R}))\geq\mathfrak{c}>\aleph_0=c(\mathbb{R})$.
\end{example}

%------------------------------------------------------------%
In \cite{MR4491569}, it is shown that for a regular space $X$ and a nontrivial $T_1$-space $Y$, we have $s(X)\leq s(Q_p(X,Y))$. In fact, a more general result holds, which is given .
%------------------------------------------------------------%
\begin{theorem}\label{thm: s(X^n)<s(Q_p)^n}
Let $X$ be a regular space and $Y$ be a nontrivial $T_1$-space. Then $s(X^{n})\le s((Q_p(X,Y))^{n})$ for each $n\in\mathbb{N}$.
\end{theorem}

\begin{proof}
Let $n\in\mathbb{N}$ be fixed and let $D$ be a discrete subspace of $X^{n}$. Then for each $d=(d_{1},\dots,d_{n})\in D$, there exists an open subset $U_{d}=U_{d,1}\times\cdots\times U_{d,n}$ of $X^{n}$ such that and $U_{d}\cap D=\{d\}$. Suppose that $y_0$ and $y_1$ are any two distinct points in $Y$ and $V$ is an open subset of $Y$ such that $y_0 \in V$ and $y_1 \notin V$.
	
Since $X$ is regular, for fixed $d\in D$ and $i\in\{1,\dots,n\}$, by Lemma~\ref{lem: qcts point-closed separation} there exists a quasicontinuous function $f_{d,i}:X\to Y$ such that $f_{d,i}(d_{i})=y_0$ and $f_{d,i}(X\setminus U_{d,i})=\{y_1\}$. Let $f_{d}=(f_{d,1},\dots,f_{d,n})\in (Q_{p}(X,Y))^{n}$, and let $A=\{f_{d}:d\in D\}$. We claim that $A$ is a discrete subset of the space $(Q_{p}(X,Y))^{n}$.
	
Let $d\in D$ be fixed. For each $i\in\{1,\dots,n\}$, consider the basic open set $W_{d,i}=[d_i;V]$. Then $W_{d}=W_{d,1}\times\cdots\times W_{d,n}$ is an open neighborhood of $f_{d}$ in $(Q_{p}(X,Y))^{n}$. We show that $W_{d}\cap A=\{f_{d}\}$. Let $d'=(d'_{1},\dots,d'_{n})\in D$ with $d'\ne d$. There exists an open subset $U_{d'}=U_{d',1}\times\cdots\times U_{d',n}$ of $X^n$ such that $U_{d'}\cap D=\{d'\}$. Thus $d\notin U_{d'}$ and hence there exists $i\in\{1,\dots,n\}$ such that $d_{i}\notin U_{d',i}$. Since $d_{i}\in X\setminus U_{d',i}$, we have $f_{d',i}(d_{i})=y_1$. Thus $f_{d',i}(d_{i})\notin V$, hence $f_{d',i}\notin W_{d,i}$. Consequently, $f_{d'}\notin W_{d}$. Hence $W_{d}\cap A=\{f_{d}\}$, this proves our claim.
	
The map $d\mapsto f_{d}$ is injective. Therefore $|D|=|A|\le s((Q_{p}(X,Y))^{n})$. Since $D$ is an arbitrary discrete subspace of $X^{n}$, it follows that $s(X^{n})\le s((Q_{p}(X))^{n})$.
\end{proof}
%------------------------------------------------------------%
\begin{remark}\label{rem: s(Q_p(R))}
Using the discrete subset $D$ in Example~\ref{exm: discrete D in Q_R}, we have $s(Q_p(\mathbb{R},\mathbb{R}))\geq \mathfrak{c}$. By \cite[Theorem 4.9]{MR4116833}, we have $w(Q_p(\mathbb{R},\mathbb{R}))=\mathfrak{c}$. It follows that $s(Q_p(\mathbb{R},\mathbb{R}))\leq w(Q_p(\mathbb{R},\mathbb{R}))= \mathfrak{c}$. Thus, $s(Q_p(\mathbb{R},\mathbb{R}))=\mathfrak{c}$.
\end{remark}
%------------------------------------------------------------%
The equality in Theorem~\ref{thm: s(X^n)<s(Q_p)^n}, that is, $s(X^{n})= s((Q_p(X,Y))^{n})$, need not hold.
\begin{example}\label{exm: s(X^n)=s(Q_p)^n fail}
Let $D$ be the  discrete subset in Example~\ref{exm: discrete D in Q_R}.
For each $n\in\mathbb{N}$, let $D_n=\{(f_a,\mathbf{0},\ldots,\mathbf{0}):a\in\mathbb{R}\}\subseteq (Q_p(\mathbb{R},\mathbb{R}))^n$, where $f_a\in D$ and $\mathbf{0}$ denotes the constant zero function. Since $D$ is discrete in $Q_p(\mathbb{R},\mathbb{R})$, for every $a\in\mathbb{R}$ there exists an open set $\mathcal{N}_a$ in $Q_p(\mathbb{R},\mathbb{R})$ such that $\mathcal{N}_a\cap D=\{f_a\}$. Let $\mathcal{W}_a=\mathcal{N}_a\times Q_p(\mathbb{R},\mathbb{R})\times\cdots\times Q_p(\mathbb{R},\mathbb{R})$. Then $\mathcal{W}_a$ is  open in $(Q_p(\mathbb{R},\mathbb{R}))^n$ and $(f_a,\mathbf{0},\ldots,\mathbf{0})\in\mathcal{W}_a$. If $(f_b,\mathbf{0},\ldots,\mathbf{0})\in\mathcal{W}_a\cap D_n$, then $f_b\in\mathcal{N}_a$, and hence $f_b=f_a$, which implies $b=a$. Therefore, $\mathcal{W}_a\cap D_n=\{(f_a,\mathbf{0},\ldots,\mathbf{0})\}$. Thus, $D_n$ is a discrete subspace of $(Q_p(\mathbb{R},\mathbb{R}))^n$. Hence, $s((Q_p(\mathbb{R},\mathbb{R}))^{n})\geq |D_n|=\mathfrak{c}$, while $s(\mathbb{R}^n)=\aleph_0$ for each $n\in\mathbb{N}$.
\end{example}
%------------------------------------------------------------%
\begin{theorem}\label{thm: s(omega)<s(Q_p)}
Let $X$ be a regular $T_1$-space and $Y$ be an infinite $T_1$-space. Let $\Omega_n(X)=\{(x_1,\dots,x_n)\in X^n: x_i\neq x_j \text{ whenever } i\neq j\}$. Then $s(\Omega_n(X))\le s(Q_p(X,Y))$ for every $n\in\mathbb{N}$.
\end{theorem}

\begin{proof}
Let $D$ be a discrete subset of the subspace $\Omega_n(X)$. For each $d=(d_1,\dots,d_n)\in D$ there exists an open subset $O_d$ of $X^n$ such that $O_d\cap D=\{d\}$. Then there exists a basic open subset $B_d=B_{d,1}\times\cdots\times B_{d,n}$ of $X^n$ such that $d\in B_d\subseteq O_d$. Since $d\in\Omega_n(X)$, $d_1,\dots,d_n$ are distinct points in $X$. Since $X$ is regular, we can choose pairwise disjoint open sets $U_{d,1},\dots,U_{d,n}$ such that $d_i\in U_{d,i}\subseteq\overline{U_{d,i}}\subseteq B_{d,i}$ for each $i\in\{1,\dots,n\}$. Let $H_d=\overline{U_{d,1}}\cup\dots\cup\overline{U_{d,n}}$. For $x\in H_d$, let $k=\min\{i\in\{1,\dots,n\}:x\in \overline{U_{d,i}}\}$. Note that $d_i\notin \overline{U_{d,j}}$ for $i\neq j$. Choose $y_0,y_1,\dots,y_n$ distinct points in $Y$.
	
For each $d\in D$, let $g_d:X\to Y$ be a function such that
\[g_d(x)=
	\begin{cases}
		y_k&\text{if } x\in H_d\\
		y_0&\text{if } x\in X\setminus H_d\\
	\end{cases}
\]
By Lemma~\ref{lem: sufficient_qcts}, each $g_d$ is quasicontinuous. For each $i\in\{0,1,\dots,n\}$, let $G_i=Y\setminus(\{y_0,y_1,\dots,y_n\}\setminus\{y_i\})$. Since $Y$ is a $T_1$-space, each $G_i$ is an open subset of $Y$. Note that $y_i\in G_i$ and for each $j\in\{0,1,\dots,n\}$ whenever $j\neq i$, $y_j\notin G_i$ . We claim that  $\{g_d:d\in D\}$ is a discrete subset of $Q_p(X,Y)$.
	
Let $N_d=[d_1,\dots,d_n;G_1,\dots,G_n]$. For each $i\in\{1,\dots,n\}$, $d_i\in U_{d,i}$ and $g_d(d_i)=y_i\in G_i$. Thus, $N_d$ is open in $Q_p(X,Y)$ and $g_d\in N_d$. Let $d'=(d'_1,\dots,d'_n)\in D$ with $d'\neq d$. Since $B_{d'}\subseteq O_{d'}$ and $O_{d'}\cap D=\{d'\}$, we have $d\notin B_{d'}$. Therefore there exists $i\in\{1,\dots,n\}$ such that $d_i\notin B_{d',i}$. Since $\overline{U}_{d',i}\subseteq B_{d',i}$, it follows that $d_i\notin \overline{U}_{d',i}$. If $d_i\in \overline{U}_{d',j}$ for some $j\neq i$, then $g_{d'}(d_i)=y_k\notin G_i$ for some $k\neq i$. If $d_i\notin\bigcup_{k=1}^n \overline{U}_{d',k}$, then $g_{d'}(d_i)=y_0\notin G_i$. Thus, $g_{d'}(d_i)\notin G_i$, so $g_{d'}\notin [d_i;G_i]$, and therefore $g_{d'}\notin N_d$. Thus $N_d\cap\{g_e:e\in D\}=\{g_d\}$ for every $d\in D$, which shows that $\{g_d:d\in D\}$ is a discrete subspace of $Q_p(X,Y)$. The map $d\mapsto g_d$ is therefore injective, so $|D|=|\{g_d:d\in D\}|\le s(Q_p(X,Y))$. Since $D$ was an arbitrary discrete subspace of $\Omega_n(X)$, it follows that $s(\Omega_n(X))\le s(Q_p(X,Y))$.
\end{proof}

%%%%%%%%%%%%%%%%%%%%%%%%%%%%%%%%%%%%%%%%%%%%%%%%%%%%%%%%%%%%%%
% Bibliography
%%%%%%%%%%%%%%%%%%%%%%%%%%%%%%%%%%%%%%%%%%%%%%%%%%%%%%%%%%%%%%

%%%%%%%%%%%%%%%%%%%%%%%%%%%%%%%%%%%%%%%%%%%%%%%%%%%%%%%%%%%%%%
\end{document}